\documentclass{amsart}

\usepackage{amsfonts,amssymb,amsmath}
\usepackage{amsthm}

\usepackage{color}
\usepackage{graphicx}

\def\R{\mathbb{R}}

\DeclareMathOperator\Hess{Hess}
\DeclareMathOperator\proj{proj}

\def\clos{\overline}

\def\tensor{\otimes}

\def\<{\left\langle}
\def\>{\right\rangle}

\begin{document}

\theoremstyle{plain}
\newtheorem{theorem}{Theorem}
\newtheorem{lemma}[theorem]{Lemma}
\newtheorem*{lm}{Lemma}
\newtheorem{corollary}[theorem]{Corollary}
\newtheorem{proposition}[theorem]{Proposition}

\theoremstyle{definition}
\newtheorem{defn}[theorem]{Definition}
\newtheorem{exmp}{Example}

\theoremstyle{remark}
\newtheorem{blah}[theorem]{}

\title[A maximum principle for pointwise energies]{A maximum principle for pointwise energies of quadratic Wasserstein minimal networks}   
\author{Jonathan Dahl}         
\address{Department of Mathematics, University of California, Berkeley, CA}
\email{jdahl@math.berkeley.edu}
\thanks{The author was supported
by the NSF under RTG grant DMS-0838703.}
\date{}    

\begin{abstract}
We show that suitable convex energy functionals on a quadratic Wasserstein space satisfy a maximum principle on minimal networks. We explore consequences of this maximum principle for the structure of minimal networks.
\end{abstract}
\maketitle

\section{Introduction}

Given $k$ points $p_1,\dots,p_k$ in a geodesic space\footnote{A geodesic space is a metric space with an intrinsic metric such that any two points are joined by a shortest path.} $Y$, one can ask for a minimal network spanning $p_1,\dots,p_k$. For a complete, connected Riemannian manifold $M$, the space of Borel probability measure $P(M)$ may be metrized, allowing infinite distances, by the Wasserstein distance $W_2$ associated to the optimal transport problem for quadratic cost $c=d^2$. Furthermore, as shown by McCann \cite{McCann01}, the component $P_2(M)$ containing the point masses is a geodesic space. Continuing the investigations in \cite{SPOT}, we examine the structure of such $W_2$-minimal networks by considering gradient flows of suitable convex energy functionals on $P_2(M)$. In particular, we establish conditions under which the energy is maximized at a boundary measure of a minimal network. The book of Ambrosio, Gigli, and Savar\'e \cite{AGS05} is an excellent source for both the history and general results of gradient flows in $P_2(\R^n)$; Villani's book \cite{Villani07} serves a parallel role for optimal transport problems.

For $M=\R^n$, we will use the fact that each minimal network in $P_2(\R^n)$ is also a solution of an associated multi-marginal optimal transportation problem. A recent result of Pass \cite{Pass10} reduces uniqueness of solutions of multi-marginal problems to computations in terms of the cost function. We will perform these computations assuming that the underlying graph of the minimal network is a star and at least one boundary point is absolutely continuous, and use this uniqueness to show that any non-expansive map fixing the boundary points must fix the network.

We begin by defining Wasserstein space, the multi-marginal optimal transportation problems, and the minimal network problems. In Section \ref{ede}, we examine the special case of differential entropy in Euclidean space. We then generalize our maximum principle to an abstract class of energy functionals in Section \ref{mpfae}. Finally, in Section \ref{armmp}, we strengthen our maximum principle in the special case of minimal stars in $P_2(\R^n)$ by examining the corresponding multi-marginal optimal transportation problem.

\section{Background}

First, we introduce the multi-marginal problem, for comparison with our quadratic Wasserstein minimal network problem defined at the end of the section. 
Given a collection of spaces $X_1,\dots, X_m$ and a collection of subsets of probability measures
$\mathcal{P}_1\subset P(X_1),\dots,\mathcal{P}_m\subset P(Y_m)$, we define the set of
transport plans $\Pi(\mathcal{P}_1,\dots,\mathcal{P}_m)$ as the set of probability measures
$\pi\in P(X_1\times\dots\times X_m)$ such that $(\proj_{X_i})_\#\pi\in\mathcal{P}_i$ for all $i$. The multi-marginal Kantorovich problem for an infinitesimal cost function $c:X_1\times\dots\times X_m\to\R$ and boundary data $\mu_1\in P(X_1),\dots,\mu_m\in P(X_m)$ is then to minimize
\[
C(\pi):=\int_{X_1\times\dots\times X_m} c(x_1,\dots,x_m)\,d\pi
\]
over all transport plans $\pi\in\Pi(\mu_1,\dots,\mu_m)$. A minimizing transport plan is called an optimal transport plan.
This is a natural relaxation of a corresponding multi-marginal Monge problem, in which one must minimize
\[
C(T_2,\dots,T_m):=\int_{X_1} c(x_1,T_2(x_1),\dots,T_m(x_1))\,d\mu_1
\]
over all $(m-1)$-tuples of measurable maps $T_i:X_1\to X_i$  such that $(T_i)_\#\mu_1=\mu_i$ for all $i$. For each $(T_2,\dots,T_m)$, the corresponding transport plan $(T_2,\dots,T_m)_\#\mu_1$ is called deterministic.

In the case $m=2$, we recover the standard Kantorvich and Monge problems. We are particularly interested in the case $X_1=X_2=X$ with $c(x_1,x_2)=d^2(x_1,x_2)$ for a metric $d$ on $X$. Taking the square root of the minimal cost, we define
\[
W_2(\mu_1,\mu_2)=\left(\inf_{\pi\in\Pi(\mu_1,\mu_2)}\int_{X^2}d^2(x_1,x_2)\,d\pi\right)^{1/2}.\]
 Assuming $(X,d)$ is complete, separable, locally compact geodesic space, $W_2$ is a metric on
\[
P_2(X):=\left\{\mu\in P(X) : \int_X d^2(x,x_0)\,d\mu
<\infty \right\},
\]
called the Wasserstein distance of order $2$, and $(P_2(X),W_2)$ is a geodesic space (\cite{McCann01} for the Riemannian case, Corollary 7.22 of \cite{Villani07} for the general case).\footnote{Here $x_0\in X$ is an arbitrary point.}

If $c(x_1,\dots,x_m)=\sum_{i<j} c_{ij}(x_i,x_j)$, the cost $C(\pi)$ of a transport plan $\pi$ can be interpreted as the total cost of the network of transport plans $(\proj_{X_i}\times\proj_{X_j})_\#\pi$ with respect to the infinitesimal costs $c_{ij}$. If $X_1=\dots=X_m=X$, for some fixed $\hat c:X^2\to\R$ each $c_{ij}$ is either $\hat c$ or identically 0, and the $\hat c$-Kantorovich problem happens to define a metric $\hat d$ on $P(X)$ where $(P(X),\hat d)$ is a geodesic space, then for a $c$-optimal transport plan $\pi$, $C(\pi)$ is precisely the $\hat d$-length of the corresponding network in $(P(X),\hat d)$. If we then remove the constraint on the last $m-k$ marginals of $\pi$, the problem of minimizing
\[
C(\pi)=\int_{X^m} c(x_1,\dots,x_m)\,d\pi
\]
over all transport plans 
\[
\pi\in\Pi(\mu_1,\dots,\mu_k,P(X),\dots,P(X))
\]
 is equivalent to the problem of finding a length minimizing network in $(P(X),\hat d)$ in the class of networks parameterized by a specific graph $G$ of $m$ vertices $v_1,\dots,v_m$ with $v_1,\dots,v_k$ mapping to $\mu_1,\dots,\mu_k$ respectively. Here, two vertices $v_i,v_j$ in $G$, $i<j$, are adjacent if and only if
$c_{ij}=\hat c$.

The formulation just mentioned will not work for the case where $(P_2(X),W_2)$ is a geodesic space, so we consider now the standard minimal network problems for a metric space. We  define a network in a metric space $Y$ as a continuous map $\Gamma:G\to Y$, where $G$ is a connected graph metrized to have edges of length 1. All graphs $G$ in what follows are connected unless stated otherwise. Given $k$ vertices in a graph $G$ and $k$ points $p_1,\dots,p_k\in Y$, the $G$-parameterized minimal network problem is to find a network $\Gamma$ with image of minimal length subject to the constraints $\Gamma(v_1)=p_1,\dots,\Gamma(v_k)=p_k$. We  define also the general minimal network problem for $p_1,\dots,p_k$: find a network $\Gamma:G\to Y$ of minimal length subject to the constraint $p_1,\dots,p_k\in \Gamma(G)$. A solution $\Gamma$ is called a minimal network. Note that $G$ is allowed to vary in the general problem. Our main focus will be on minimal network problems for $Y=(P_2(X),W_2)$, a quadratic Wasserstein space.

\section{Euclidean differential entropy}\label{ede}

Under certain geometric assumptions on the base space $X$, regularity of the vertices of a solution of the general minimal network problem in $(P_2(X),W_2)$ can lead to combinatorial regularity of the solution. In particular,
\begin{theorem}[\cite{SPOT}]\label{str}
Suppose $M$ is a Riemannian manifold with isometric splitting
$M=S\times\mathbb{R}^n$ where $S$ is compact with nonnegative
sectional curvature, and $\mu_1,\dots,\mu_k\in P_2(M)$ have
compact support. Then there is a minimal network
in $P_2(M)$ spanning $\mu_1,\dots,\mu_k$. Furthermore, this solution
has a canonical representative $\Gamma:G\to P_2(M)$ such that
\begin{enumerate}
    \item $\Gamma$ is injective.
    \item $G$ is a tree.
    \item Vertices in $G$ not mapped to $\mu_1,\dots,\mu_k$
    have degree at least three.
    \item For any vertex $v$ in $G\setminus\Gamma^{-1}(\{\mu_1,\dots,\mu_k\})$,
    if $\Gamma(v)$ is absolutely continuous with respect to the volume measure,
    then the degree of $v$ is three and all pairs of geodesics in $\Gamma(G)$
    meeting at $\Gamma(v)$ do so with an angle of $2\pi/3$.
\end{enumerate}
\end{theorem}
This result follows from the inner product structure of the tangent cone at $\nu\in P_2^{ac}(M)$ demonstrated by Lott and Villani \cite{LV04} for $\nu$ of compact support. In the special case $M=\R^n$, we may therefore remove the compact support hypothesis \cite{AGS05}.
It is natural to ask then if absolute continuity (w.r.t. Lebesgue measure $\mathcal{L}^n$) of each measure $\mu_i$ in the boundary data implies absolute continuity for the Steiner solution, i.e. $\Gamma(G)\subset P_2^{ac}(\R^n)$. In this direction, we recall the definition of Shannon's \cite{Shannon48} differential entropy
$h:P_2(\R^n)\to [-\infty,+\infty)$,
\[
h(\mu)=
\begin{cases}
-\int_{\R^n} \rho(x)\log\rho(x)\,d\mathcal{L}^n(x) & \text{if $\mu=\rho \mathcal{L}^n$,}\\
-\infty & \text{otherwise.}
\end{cases}
\]

\begin{proposition}\label{babycase}
If $\mu_1,\dots,\mu_k\in P_2^{ac}(\R^n)$ are taken as boundary data for a graph $G$ and $m=\min_i h(\mu_i)>-\infty$, then there is a solution of the $G$-parameterized minimal network problem contained entirely in $P_2^{ac}(\R^n)$.
\end{proposition}
\begin{proof}
\[
\mathcal{F}(\mu)=\max(-m,-h(\mu))
\]
defines a convex functional $\mathcal{F}:P_2(\R^n)\to(-\infty,+\infty]$ with effective domain
\[
D(\mathcal{F}):=\mathcal{F}^{-1}(\R)\subset P_2^{ac}(\R^n).
\]
 As in Theorem 11.2.1 of \cite{AGS05}, we obtain a non-expansive gradient flow for $\mathcal{F}$ on $\clos{D(\mathcal{F})}=P_2(\R^n)$.\footnote{This is simply the gradient flow for $-h$, or heat flow, stopped at $m$.}

The flow is instantly regularizing, so the flowed network lies in $P_2^{ac}(\R^n)$ for all $t>0$. Furthermore, $\mathcal{F}$ is minimized and constant on
each $\mu_i$ in the boundary data, so they remain fixed under the flow. Therefore, if we start with an injective solution $\Gamma:G\to P_2(\R^n)$ of the $G$-parameterized minimal network problem \cite{SPOT}, non-expansiveness of the flow provides a solution $\Gamma_t:G\to P_2^{ac}(\R^n)$ for $t>0$.
\end{proof}

Recall that the Steiner ratio of a metric space $X$ is
    \[\inf_M \frac{L_s(M)}{L_a(M)},\]
    where $M$ is any finite set of points in $X$,
    $L_s(M)$ is the infimum of the lengths of all networks
    spanning $M$, and $L_a(M)$ is the length of the
    minimal spanning tree of $M$.

\begin{corollary}\label{euclidsteinerratio}
The Steiner ratio of $P_2(\R^n)$ is at least $1/\sqrt3$.
\end{corollary}
\begin{proof}
$P_2(\R^n)$ is an Alexandrov space on nonnegative curvature
(see \cite{LV04} or \cite{MMS1} for stronger statements). Propostion \ref{babycase}, Theorem \ref{str}, and Toponogov's theorem for Alexandrov spaces allow us to apply the arguments of Graham and Hwang \cite{GH76} to solutions of the general minimal network problem, provided the boundary data $\mu_1,\dots\mu_k\in P_2^{ac}(\R^n)$ have finite differential entropy. Approximating arbitrary boundary data by finite linear combinations of characteristic functions of balls in $\R^n$, we may bound the Steiner ratio.
\end{proof}

\section{Maximum principle for admissible energies}\label{mpfae}

In this section, we will introduce a suitable generalization of (negative) differential entropy for which a maximum principle on minimal networks holds. This will allow us to obtain greater control of the structure of minimal networks.

\begin{defn}
A proper, lower semi-continuous functional $\phi:P_2(X)\to(-\infty,+\infty]$ for a space $X$ is called admissible if:
\begin{enumerate}
\item $D(\phi)=\phi^{-1}(\R)$ is a geodesic space, and $\phi$ is convex along geodesics in $D(\phi)$.
\item For each $a\in[-\infty,+\infty)$, $t>0$, there is a non-expansive map $\psi_{a,t}:P_2(X)\to D(\phi)$, called the $a$-stopped flow of $\phi$ at $t$.
\item The set of fixed points for $\psi_{a,t}$ is $\phi^{-1}(-\infty,a]$.
\end{enumerate}
\end{defn}

We are thinking of $\psi_{a,t}$ as the gradient flow for $\phi$ stopped at the value $a$. The usual convexity condition for such functionals is weak displacement convexity, where convexity is only assumed along some geodesic between every two points in $P_2(X)$. It may therefore be more natural to consider the convexity condition as weak displacement convexity combined with uniqueness of geodesics in $D(\phi)$, even though the condition in our definition is strictly weaker.

More importantly, we require the flow to exist on all of $P_2(X)$, and hence want either $D(\phi)=P_2(X)$ or $\clos{D(\phi)}=P_2(X)$ with an appropriate regularization estimate in order to  obtain $\psi_{a,t}(\clos{D(\phi)})\subset D(\phi)$. Existence of such regularizing gradient flows has been announced by Savar\'e \cite{Savare07}, under mild geometric assumptions on $X$. In particular, $X$ can be any complete Riemannian manifold, or Alexandrov space, with curvature bounded below and finite diameter. Similar results were obtained independently by Ohta for any
proper, lower semi-continuous convex functional $\phi:P_2(X)\to(-\infty,+\infty]$ for a compact Alexandrov space $X$ \cite{Ohta09}, although here we do not necessarily have $\psi_{a,t}(\clos{D(\phi)})\subset D(\phi)$. In the particular case of relative entropy w.r.t.\ volume measure, admissibility of $\phi$ is shown by Erbar \cite{Erbar10} for any connected, complete Riemannian manifold $X$ with a lower bound on Ricci curvature.

We now present our maximum principle.

\begin{theorem}\label{maxprinc}
Suppose $X$ is a complete, separable, locally compact, non-branching geodesic space and let $\phi:X\to(-\infty,+\infty]$ be an admissible functional.
For $\mu_1,\dots\mu_k\in P_2(X)$ and $m=\max_i \phi(\mu_i)$, any solution $\Gamma:G\to P_2(X)$ of the $G$-parameterized minimal network problem for boundary data $\mu_1,\dots,\mu_k$ has
\[
\Gamma(G)\subset \phi^{-1}(-\infty,m].
\]
\end{theorem}
\begin{proof}
If $m=\infty$, there is nothing to prove, so assume $m<\infty$. Let $\Gamma_t=\psi_{m,t}\circ\Gamma$. $\psi_{m,t}$ is non-expansive and $\mu_1,\dots,\mu_k$ are fixed by $\psi_{m,t}$, so $\Gamma_t$ also solves the $G$-parameterized minimal network problem. In particular, the image of each edge is a geodesic, and $\psi$ is continuous on $\Gamma_t(G)$ by convexity. Fix $t_0>0$. Let $M=\max_{\Gamma_{t_0}(G)}\phi$, and suppose that $M>m$. $G$ is a finite graph, so we may choose $a\in(m,M)$ such that $\phi\neq a$ on vertices of $\Gamma_{t_0}(G)$.

Consider $\psi_{a,s}\circ\Gamma_{t_0}$. For some edge $e_1$ of $\Gamma_{t_0}$, $\psi\le a$ on an initial interval $e'_1$ of $e_1$ while $\psi>a$ at the opposite endpoint $v_2$.  As before, $\psi_{a,s}(e_1)$ is a geodesic, but $\psi_{a,s}(e'_1)=e'_1$ while $\psi_{a,s}(v_2)\neq v_2$. We therefore have a branched geodesic in $P_2(X)$, contradicting the fact that $X$ is non-branching by Corollary 7.32 of \cite{Villani07}. Thus $M=m$.
\end{proof}

\begin{corollary}\label{finengstr}
Suppose $M$ is a Riemannian manifold with isometric splitting
$M=S\times\mathbb{R}^n$ where $S$ is compact with nonnegative
sectional curvature, $\mu_1,\dots,\mu_k\in P_2(M)$ have
compact support, and $\max_i \phi(\mu_i)<+\infty$ for an admissible functional $\phi$. Then there is a minimal network
in $D(\phi)$ spanning $\mu_1,\dots,\mu_k$. Furthermore,
if $D(\phi)\subset P_2^{ac}(M)$,
this solution
has a canonical representative $\Gamma:G\to P_2(M)$ such that
\begin{enumerate}
    \item $\Gamma$ is injective.
    \item $G$ is a tree.
    \item Vertices in $G$ not mapped to $\mu_1,\dots,\mu_k$
    have degree three.
    \item For any vertex $v$ in $G\setminus\Gamma^{-1}(\{\mu_1,\dots,\mu_k\})$,
    all pairs of geodesics in $\Gamma(G)$
    meeting at $\Gamma(v)$ do so with an angle of $2\pi/3$.
\end{enumerate}
\end{corollary}

Again, for $M=\R^n$, compactness of supports is unnecessary. Typical candidates for $\phi$ with $D(\phi)\subset P_2^{ac}(\R^n)$ are internal energies, such as differential entropy, relative entropy with respect to an absolutely continuous, log-concave reference measure, and
the power functional
\[
\mathcal{F}(\mu)=
\begin{cases}
\int_{\R^n} (\rho(x))^m\,d\mathcal{L}^n(x) & \text{if $\mu=\rho \mathcal{L}^n$,}\\
+\infty & \text{otherwise,}
\end{cases}
\]
for some fixed $m>1$.


Applying our work again to the special case of finite linear combinations of characteristic functions of balls, we may generalize Corollary \ref{euclidsteinerratio}.

\begin{corollary}\label{steinerratio}
If $M$ is a Riemannian manifold with isometric splitting
$M=S\times\mathbb{R}^n$ where $S$ is compact with nonnegative
sectional curvature,
then the Steiner ratio of $P_2(M)$ is at least $1/\sqrt3$.
\end{corollary}
\begin{proof}
We may clearly reduce to the compact case by projection in $M$, where relative entropy
\[
\mathcal{F}(\mu)=
\begin{cases}
\int_{M} \rho(x)\log \rho(x)\,d\gamma & \text{if $\mu=\rho \gamma$,}\\
+\infty & \text{otherwise,}
\end{cases}
\]
for the volume measure $\gamma$ defines an admissible functional as in \cite{Erbar10} or \cite{Savare07}.
The result now follows as before for Corollary \ref{euclidsteinerratio}.
\end{proof}

\section{A related multi-marginal problem}\label{armmp}

We now motivate a related problem in the uniqueness theory of multi-marginal problems.
Let $l\ge3$, and consider the $G$-parameterized minimal network problem of minimizing
\[
\Phi(\nu):=\sum_{i=1}^l W_2(\mu_i,\nu)
\]
for some fixed boundary data $\mu_1,\dots,\mu_l\in P_2(\R^n)$, where $G$ is the star with $l$ leaves. Suppose $\nu_0$ minimizes $\Phi$, and assume we are in the nontrivial case
\[
W_2(\mu_i,\nu_0)>0,\quad 1\le i\le l.
\]

For
\[
\sigma_i:=1/W_2(\nu_0,\mu_i),
\]
$\nu_0$ is also a minimizer of
\[
\Psi(\mu):=\sum_{i=1}^l \sigma_i W_2^2(\mu_i,\mu),
\]
as the corresponding $l$-plane and ellipsoid are tangent.

\begin{equation*}
\begin{split}
\Psi(\mu)
&=\sum_{i=1}^l \sigma_i
\inf_{\pi\in\Pi(\mu_i,\mu)}\int_{\R^{2n}}|x_1-x_2|^2\,d\pi\\
&=
\inf_{\pi\in\Pi(\mu_1,\dots,\mu_l,\mu)}\int_{\R^{(l+1)n}}
\sum_{i=1}^l \sigma_i |x_i-x_j|^2\,d\pi,
\end{split}
\end{equation*}
so
\[
\inf_{\nu\in P_2(\R^n)}\Psi(\nu)
=\inf_{\pi\in\Pi(\mu_1,\dots,\mu_l,P_2(\R^n))}\int_{\R^{(l+1)n}}
\sum_{i=1}^l \sigma_i |x_i-x_{l+1}|^2\,d\pi.
\]

For simplicity, suppose that $\mu_1,\dots,\mu_l\in P_2(\R^n)$ have compact support,
take $R$ large enough that the support of each $\mu_i$ is contained in $B_R(0)$, and set $B=\clos{B_{R+1}(0)}$.

By standard duality arguments, the infimum is achieved and clearly supported in $B$, with
\[
\min_{\pi\in\Pi(\mu_1,\dots,\mu_l,P_2(B))}\int_{B^{l+1}}
\sum_{i=1}^l \sigma_i |x_i-x_{l+1}|^2\,d\pi
=\sup_{(\phi_1,\dots,\phi_{l+1})\in S}
\left(\sum_{i=1}^{l}\int_{B}\phi_i\,d\mu_i\right)-\|\phi_{l+1}\|_\infty,
\]
for
\[
S=\left\{(\phi_1,\dots,\phi_{l+k})\in (C_b(B))^{l+1}:
\forall x_1,\dots,x_{l+1}\in B, \sum_{i=1}^{l+1} \phi_i(x_i)\le
\sum_{i=1}^{l} \sigma_i |x_i-x_{l+1}|^2 \right\}.
\]
If we let
\[
\bar x:=\left(\sum_{i=1}^l \sigma_i x_i\right)\bigg/\left(\sum_{i=1}^l \sigma_i \right)
\]
 and
\[
S'=\left\{(\phi_1,\dots,\phi_l) \in (C_b(B))^{l}:
\forall x_1,\dots,x_l\in B, \sum_{i=1}^{l}\phi_i(x_i)\le\sum_{i=1}^{l}\sigma_i|x_i-\bar x|^2\right\},
\]
we may suitably adjust any $(\tilde\phi_1,\dots,\tilde\phi_{l+1})\in S$ to obtain an element of $S'$ by
taking
\[
\phi_1=\tilde\phi_1+M, \phi_2=\tilde\phi_2, \dots, \phi_l=\tilde\phi_l
\]
for
\[
M=\min\left(0,\, \min_{x\in B} \tilde\phi_{l+1}(x)\right).
\]
$(\phi_1,\dots,\phi_l,0)\in S$ by construction and $\|0\|_\infty=0$, so
\[\sup_{(\phi_1,\dots,\phi_{l+1})\in S}
\left(\sum_{i=1}^{l}\int_{B}\phi_i\,d\mu_i\right)-\|\phi_{l+1}\|_\infty
=\sup_{(\phi_1,\dots,\phi_l) \in S'}
\left(\sum_{i=1}^{l}\int_{B}\phi_i\,d\mu_i\right).
\]
The right hand side is also dual for the multi-marginal Kantorovich problem with infinitesimal cost
$c(x_1,\dots,x_l)=\sum_{i=1}^{l}\sigma_i |x_i-\bar x|^2$, so
\[
\min_{\pi\in\Pi(\mu_1,\dots,\mu_l,P_2(B))}\int_{B^{l+1}}
\sum_{i=1}^l \sigma_i |x_i-x_{l+1}|^2\,d\pi
=
\min_{\pi\in\Pi(\mu_1,\dots,\mu_l)}\int_{B^{l}}
\sum_{i=1}^l \sigma_i |x_i-\bar x|^2\,d\pi.
\]
We may also see this directly by noting that the map $f(x_1,\dots,x_l)=\bar x$ satisfies
\[
\int_{B^{l+1}}
\sum_{i=1}^l \sigma_i |x_i-x_{l+1}|^2\,d(\pi\times f_\#\pi)
=
\int_{B^{l}}
\sum_{i=1}^l \sigma_i |x_i-\bar x|^2\,d\pi
\]
for all $\pi\in\Pi(\mu_1,\dots,\mu_l)$, and any minimizer in
$\Pi(\mu_1,\dots,\mu_l,P_2(B))$
must have this form by the minimizing property of $f$.\footnote{In fact, the equality continues to hold for $B=\R^n$, arbitrary $\mu_1,\dots,\mu_l\in P_2(\R^n)$, and arbitrary graphs $G$, with the minimum on the left achieved by our assumed solution of the minimal network problem,
and the minimum on the right achieved by a solution of the Kantorovich problem
\cite{Kellerer84}.} In particular, uniqueness of the minimizer $\nu_0$ of $\Psi$ would follow from uniqueness of a minimizer for the multi-marginal Kantorovich problem for $\mu_1,\dots,\mu_l$ and $c(x_1,\dots,x_l)=\sum_{i=1}^{l} \sigma_i |x_i-\bar x|^2$. Furthermore, such uniqueness would imply that $\nu=\nu_0$ is the unique solution of the system
\[
W_2(\mu_1,\nu)=W_2(\mu_1,\nu_0),\dots,W_2(\mu_l,\nu)=W_2(\mu_l,\nu_0).
\]
This would give a result similar to our maximum principle above, without referencing the potential $\phi$ of the non-expansive flow. In fact, any non-expansive map fixing the boundary data $\mu_1,\dots,\mu_l$ would be forced to fix $\nu_0$ as well.

\begin{proposition}\label{absmaxprinc}
The multi-marginal Kantorovich problem for $\mu_1,\dots,\mu_l\in P_2(\R^n)$ with compact support, $\mu_1\in P_2^{ac}(\R^n)$, and
\[
c(x_1,\dots,x_l)=\sum_{i=1}^{l} \sigma_i |x_i-\bar x|^2
\]
 has a unique solution for any collection of positive weights $\sigma_i$. Hence, any non-expansive map $F:P_2(\R^n)\to P_2(\R^n)$ with $F(\mu_1)=\mu_1,\dots,F(\mu_l)=\mu_l$ must fix the free vertex $\nu_0$ of a solution of the $G$-parameterized minimal network problem, where $G$ is the star whose $l$ leaves are the fixed vertices.
\end{proposition}
\begin{proof}
For $l=2$, the result is simply geodesic uniqueness and follows from Brenier's theorem. Assume $l\ge3$.

The following conditions for infinitesimal cost functions $c:M_1\times\dots\times M_m\to\R$ on precompact\footnote{Precisely, each $M_i$ is smoothly embedded in some manifold $N_i$, in which $\clos{M_i}$ is compact.} Riemannian manifolds $M_1,\dots,M_m$ of dimension $n$ were used by Pass  \cite{Pass10} to obtain uniqueness of solutions to multi-marginal problems:
\begin{enumerate}
\item $c\in C^2(\clos{M_1}\times\dots\times\clos{M_m})$.
\item $c$ is $(1,m)$-twisted, meaning the map $x_m\mapsto D_{x_1}c(x_1,\dots,x_m)$ from $M_m$ to $T_{x_1}^* M_1$ is injective for all fixed $x_k$, $k\neq m$.
\item $c$ is $(1,m)$-non-degenerate, meaning $D_{x_1x_m}^2 c(x_1,\dots,x_m):T_{x_m}M_m\to T_{x_1}^* M_1$ is injective for all $(x_1,\dots,x_m)$.
\item For all choices of $y=(y_1,\dots,y_m)\in M_1\times\dots\times M_m$ and of $y(i)=(y_1(i),\dots,y_i(m))\in \clos{M_1}\times\dots\times\clos{M_m}$ such that $y_i(i)=y_i$ for $i=2,\dots,m-1$, we have
\[
T_{y,y(2),y(3),\dots,y(m-1)}<0,
\]
for $T_{y,y(2),y(3),\dots,y(m-1)}$ defined below.
\end{enumerate}
\begin{theorem}[\cite{Pass10}]\label{Pass}
If $M_1,\dots,M_m$ and $c$ satisfy the above conditions and $\mu_1\in P(M_1)$ does not charge sets of Hausdorff dimension less than or equal to $n-1$, then the multi-marginal Kantorovich and Monge problems have unique solutions for any $\mu_2\in P(M_2),\dots,\mu_m\in P(M_m)$.
\end{theorem}

We now define $T_{y,y(2),y(3),\dots,y(m-1)}$:
\begin{equation*}
\begin{split}
S_y:=&-\sum_{j=2}^{m-1}\sum_{i=2, i\neq j}^{m-1}D_{x_ix_j}^2c(y)\\
&\quad+\sum_{i,j=2}^{m-1}(D_{x_ix_m}^2c(D_{x_1x_m}^2c)^{-1}D_{x_1x_j}^2c)(y),\\
H_{y,y(2),y(3),\dots,y(m-1)}:=&\sum_{i=1}^{m-1}(\Hess_{x_i}c(y(i))-\Hess_{x_i}c(y)),\\
T_{y,y(2),y(3),\dots,y(m-1)}:=&S_y+H_{y,y(2),y(3),\dots,y(m-1)}.
\end{split}
\end{equation*}

For our infinitesimal cost $c(x_1,\dots,x_l)=\sum_{i=1}^l \sigma_i |x_i-\bar x|^2$, we may compute directly in natural coordinates and find that at $(p_1,\dots,p_l)\in \R^{ln}$,
\[
D_{x_i}c(p_1,\dots,p_l)
=\sum_{\alpha=1}^n \left[2\left(\sigma_i-\frac{\sigma_i^2}{\sum_{k=1}^l \sigma_k}\right) p_i^\alpha
- \sum_{j\neq i} \frac{2\sigma_i\sigma_j}{\sum_{k=1}^l \sigma_k} p_j^\alpha\right]dx_i^\alpha,
\]
and
\[
D_{x_ix_j}c(p_1,\dots,p_l)=
\begin{cases}
\sum_\alpha 2\left(\sigma_i-\dfrac{\sigma_i^2}{\sum_{k=1}^l \sigma_k}\right)  dx_i^\alpha\tensor dx_i^\alpha & \text{if $i=j$,}\\
\sum_\alpha \dfrac{-2\sigma_i\sigma_j}{\sum_{k=1}^l \sigma_k}  dx_i^\alpha\tensor dx_j^\alpha & \text{if $i\neq j$.}
\end{cases}
\]
Therefore,
$c$ is $C^2$, $(1,l)$-twisted, and $(1,l)$-non-degenerate. $H\equiv 0$, so
\[
T=\sum_{i=2}^{l-1}
\frac{-2\sigma_i^2}{\sum_{k=1}^l \sigma_k} \sum_\alpha dx_i^\alpha\tensor dx_i^\alpha
<0.\qedhere
\]
\end{proof}

The condition $\mu_1\in P_2^{ac}(\R^d)$ cannot be removed, as seen by considering 
\[
\mu_1=\frac12 \delta_{(1,0)}+\frac12 \delta_{(-1,0)},\quad
\mu_2=\frac12 \delta_{(0,1)}+\frac12 \delta_{(0,-1)},
\]
and the isometry $F$ on $P_2(\R^n)$ induced by reflection about the $x$-axis. Only one of the geodesics from $\mu_1$ to $\mu_2$ is fixed by $F$.

Most of the preceding discussion may be generalized to an arbitrary graph $G$, where $\bar x$ is replaced by the collection of points in $\R^n$ solving the appropriately weighted minimization problem for the sum of squared distances. However, we can have $T\not<0$ in this case, causing the end of the proof of Proposition \ref{absmaxprinc} to fail. In particular, if we take a minimal network as in Figure \ref{figh}, computations similar to our previous work show that $T<0$ if and only if $a>\sqrt2b$. Switching the labeling of $\mu_2$ and $\mu_4$ only makes the inequality worse, with $T<0$ if and only if $a>4b$.

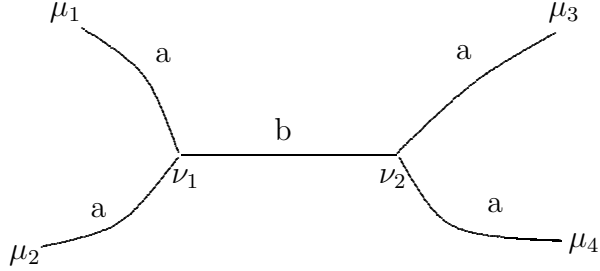
\begin{figure}
\caption{A $P_2(\R^n)$ minimal network with boundary data $\mu_1,\dots,\mu_4$ and lengths $a,b$ as labeled.}
\label{figh}
\setlength{\unitlength}{3947sp}%
\begingroup\makeatletter\ifx\SetFigFont\undefined%
\gdef\SetFigFont#1#2#3#4#5{%
  \reset@font\fontsize{#1}{#2pt}%
  \fontfamily{#3}\fontseries{#4}\fontshape{#5}%
  \selectfont}%
\fi\endgroup%
\begin{picture}(3492,1942)(496,-1320)
\thinlines
{\color[rgb]{0,0,0}\put(1591,-466){\line( 1, 0){1350}}
}%
{\color[rgb]{0,0,0}\multiput(2948,-459)(8.00000,8.00000){2}{\makebox(1.6667,11.6667){\tiny.}}
\multiput(2956,-451)(6.50000,6.50000){3}{\makebox(1.6667,11.6667){\tiny.}}
\multiput(2969,-438)(8.75000,8.75000){3}{\makebox(1.6667,11.6667){\tiny.}}
\multiput(2987,-421)(7.00000,7.00000){4}{\makebox(1.6667,11.6667){\tiny.}}
\multiput(3008,-400)(6.25000,6.25000){5}{\makebox(1.6667,11.6667){\tiny.}}
\multiput(3034,-376)(6.87500,6.87500){5}{\makebox(1.6667,11.6667){\tiny.}}
\multiput(3062,-349)(7.25000,7.25000){5}{\makebox(1.6667,11.6667){\tiny.}}
\multiput(3091,-320)(5.90000,5.90000){6}{\makebox(1.6667,11.6667){\tiny.}}
\multiput(3121,-291)(5.90000,5.90000){6}{\makebox(1.6667,11.6667){\tiny.}}
\multiput(3151,-262)(7.74590,6.45492){5}{\makebox(1.6667,11.6667){\tiny.}}
\multiput(3181,-235)(6.87500,6.87500){5}{\makebox(1.6667,11.6667){\tiny.}}
\multiput(3209,-208)(6.50000,6.50000){5}{\makebox(1.6667,11.6667){\tiny.}}
\multiput(3236,-183)(6.25000,6.25000){5}{\makebox(1.6667,11.6667){\tiny.}}
\multiput(3262,-159)(8.16393,6.80328){4}{\makebox(1.6667,11.6667){\tiny.}}
\multiput(3286,-138)(7.16667,7.16667){4}{\makebox(1.6667,11.6667){\tiny.}}
\multiput(3308,-117)(7.27870,6.06558){4}{\makebox(1.6667,11.6667){\tiny.}}
\multiput(3330,-99)(7.08197,5.90164){4}{\makebox(1.6667,11.6667){\tiny.}}
\multiput(3351,-81)(6.72130,5.60108){4}{\makebox(1.6667,11.6667){\tiny.}}
\multiput(3371,-64)(9.45120,7.56096){3}{\makebox(1.6667,11.6667){\tiny.}}
\multiput(3390,-49)(9.54100,7.95083){3}{\makebox(1.6667,11.6667){\tiny.}}
\multiput(3409,-33)(10.15385,6.76923){3}{\makebox(1.6667,11.6667){\tiny.}}
\multiput(3429,-19)(9.45120,7.56096){3}{\makebox(1.6667,11.6667){\tiny.}}
\multiput(3448, -4)(9.44000,7.08000){3}{\makebox(1.6667,11.6667){\tiny.}}
\multiput(3467, 10)(10.15385,6.76923){3}{\makebox(1.6667,11.6667){\tiny.}}
\multiput(3487, 24)(6.88000,5.16000){4}{\makebox(1.6667,11.6667){\tiny.}}
\multiput(3508, 39)(7.00000,4.66667){4}{\makebox(1.6667,11.6667){\tiny.}}
\multiput(3529, 53)(7.61540,5.07693){4}{\makebox(1.6667,11.6667){\tiny.}}
\multiput(3552, 68)(7.61540,5.07693){4}{\makebox(1.6667,11.6667){\tiny.}}
\multiput(3575, 83)(8.72550,5.23530){4}{\makebox(1.6667,11.6667){\tiny.}}
\multiput(3601, 99)(8.97060,5.38236){4}{\makebox(1.6667,11.6667){\tiny.}}
\multiput(3628,115)(7.31618,4.38971){5}{\makebox(1.6667,11.6667){\tiny.}}
\multiput(3657,133)(7.50000,4.50000){5}{\makebox(1.6667,11.6667){\tiny.}}
\multiput(3687,151)(7.68382,4.61029){5}{\makebox(1.6667,11.6667){\tiny.}}
\multiput(3718,169)(8.16178,4.89707){5}{\makebox(1.6667,11.6667){\tiny.}}
\multiput(3751,188)(8.16178,4.89707){5}{\makebox(1.6667,11.6667){\tiny.}}
\multiput(3784,207)(7.68382,4.61029){5}{\makebox(1.6667,11.6667){\tiny.}}
\multiput(3815,225)(7.90000,3.95000){5}{\makebox(1.6667,11.6667){\tiny.}}
\multiput(3846,242)(8.97060,5.38236){4}{\makebox(1.6667,11.6667){\tiny.}}
\multiput(3873,258)(7.54903,4.52942){4}{\makebox(1.6667,11.6667){\tiny.}}
\multiput(3896,271)(9.41175,5.64705){3}{\makebox(1.6667,11.6667){\tiny.}}
\multiput(3915,282)(13.82350,8.29410){2}{\makebox(1.6667,11.6667){\tiny.}}
\multiput(3929,290)(10.00000,5.00000){2}{\makebox(1.6667,11.6667){\tiny.}}
}%
{\color[rgb]{0,0,0}\multiput(2961,-490)(6.00000,-12.00000){2}{\makebox(1.6667,11.6667){\tiny.}}
\multiput(2967,-502)(4.72060,-7.86767){3}{\makebox(1.6667,11.6667){\tiny.}}
\multiput(2976,-518)(4.38237,-7.30394){4}{\makebox(1.6667,11.6667){\tiny.}}
\multiput(2989,-540)(5.29413,-8.82356){4}{\makebox(1.6667,11.6667){\tiny.}}
\multiput(3004,-567)(4.50000,-7.50000){5}{\makebox(1.6667,11.6667){\tiny.}}
\multiput(3022,-597)(4.78677,-7.97796){5}{\makebox(1.6667,11.6667){\tiny.}}
\multiput(3041,-629)(4.96323,-8.27204){5}{\makebox(1.6667,11.6667){\tiny.}}
\multiput(3061,-662)(5.42307,-8.13461){5}{\makebox(1.6667,11.6667){\tiny.}}
\multiput(3082,-695)(5.11538,-7.67306){5}{\makebox(1.6667,11.6667){\tiny.}}
\multiput(3102,-726)(4.88462,-7.32694){5}{\makebox(1.6667,11.6667){\tiny.}}
\multiput(3122,-755)(6.56000,-8.74667){4}{\makebox(1.6667,11.6667){\tiny.}}
\multiput(3142,-781)(6.28000,-8.37333){4}{\makebox(1.6667,11.6667){\tiny.}}
\multiput(3161,-806)(6.20220,-7.44264){4}{\makebox(1.6667,11.6667){\tiny.}}
\multiput(3180,-828)(6.33333,-6.33333){4}{\makebox(1.6667,11.6667){\tiny.}}
\multiput(3199,-847)(6.00000,-6.00000){4}{\makebox(1.6667,11.6667){\tiny.}}
\multiput(3217,-865)(9.45120,-7.56096){3}{\makebox(1.6667,11.6667){\tiny.}}
\multiput(3236,-880)(9.44000,-7.08000){3}{\makebox(1.6667,11.6667){\tiny.}}
\multiput(3255,-894)(9.92310,-6.61540){3}{\makebox(1.6667,11.6667){\tiny.}}
\multiput(3275,-907)(10.36765,-6.22059){3}{\makebox(1.6667,11.6667){\tiny.}}
\multiput(3296,-919)(9.13795,-3.65518){3}{\makebox(1.6667,11.6667){\tiny.}}
\multiput(3314,-927)(9.80000,-4.90000){3}{\makebox(1.6667,11.6667){\tiny.}}
\multiput(3334,-936)(10.05000,-3.35000){3}{\makebox(1.6667,11.6667){\tiny.}}
\multiput(3354,-943)(10.50000,-3.50000){3}{\makebox(1.6667,11.6667){\tiny.}}
\multiput(3375,-950)(11.52940,-2.88235){3}{\makebox(1.6667,11.6667){\tiny.}}
\multiput(3398,-956)(8.31373,-2.07843){4}{\makebox(1.6667,11.6667){\tiny.}}
\multiput(3423,-962)(8.97437,-1.79487){4}{\makebox(1.6667,11.6667){\tiny.}}
\multiput(3450,-967)(9.67567,-1.61261){4}{\makebox(1.6667,11.6667){\tiny.}}
\multiput(3479,-972)(10.32433,-1.72072){4}{\makebox(1.6667,11.6667){\tiny.}}
\multiput(3510,-977)(10.91893,-1.81982){4}{\makebox(1.6667,11.6667){\tiny.}}
\multiput(3543,-981)(8.87838,-1.47973){5}{\makebox(1.6667,11.6667){\tiny.}}
\multiput(3579,-984)(9.40540,-1.56757){5}{\makebox(1.6667,11.6667){\tiny.}}
\put(3617,-988){\line( 1, 0){ 40}}
\put(3657,-991){\line( 1, 0){ 41}}
\put(3698,-994){\line( 1, 0){ 42}}
\put(3740,-996){\line( 1, 0){ 42}}
\put(3782,-999){\line( 1, 0){ 39}}
\put(3821,-1000){\line( 1, 0){ 37}}
\put(3858,-1002){\line( 1, 0){ 33}}
\put(3891,-1003){\line( 1, 0){ 28}}
\put(3919,-1004){\line( 1, 0){ 22}}
\put(3941,-1005){\line( 1, 0){ 16}}
\multiput(3957,-1005)(10.86490,-1.81082){2}{\makebox(1.6667,11.6667){\tiny.}}
\put(3968,-1006){\line( 1, 0){  5}}
\put(3973,-1006){\line( 1, 0){  3}}
}%
{\color[rgb]{0,0,0}\multiput(1573,-456)(-4.82760,12.06900){2}{\makebox(1.6667,11.6667){\tiny.}}
\multiput(1568,-444)(-2.85000,8.55000){3}{\makebox(1.6667,11.6667){\tiny.}}
\multiput(1562,-427)(-4.58620,11.46550){3}{\makebox(1.6667,11.6667){\tiny.}}
\multiput(1553,-404)(-3.23333,9.70000){4}{\makebox(1.6667,11.6667){\tiny.}}
\multiput(1543,-375)(-3.17243,7.93106){5}{\makebox(1.6667,11.6667){\tiny.}}
\multiput(1531,-343)(-3.37930,8.44825){5}{\makebox(1.6667,11.6667){\tiny.}}
\multiput(1518,-309)(-3.50000,8.75000){5}{\makebox(1.6667,11.6667){\tiny.}}
\multiput(1504,-274)(-3.50000,8.75000){5}{\makebox(1.6667,11.6667){\tiny.}}
\multiput(1490,-239)(-3.37930,8.44825){5}{\makebox(1.6667,11.6667){\tiny.}}
\multiput(1477,-205)(-3.24138,8.10344){5}{\makebox(1.6667,11.6667){\tiny.}}
\multiput(1463,-173)(-3.93103,9.82758){4}{\makebox(1.6667,11.6667){\tiny.}}
\multiput(1450,-144)(-3.77010,9.42525){4}{\makebox(1.6667,11.6667){\tiny.}}
\multiput(1438,-116)(-4.20000,8.40000){4}{\makebox(1.6667,11.6667){\tiny.}}
\multiput(1425,-91)(-3.80000,7.60000){4}{\makebox(1.6667,11.6667){\tiny.}}
\multiput(1414,-68)(-3.73333,7.46667){4}{\makebox(1.6667,11.6667){\tiny.}}
\multiput(1402,-46)(-5.77940,9.63233){3}{\makebox(1.6667,11.6667){\tiny.}}
\multiput(1390,-27)(-5.77940,9.63233){3}{\makebox(1.6667,11.6667){\tiny.}}
\multiput(1378, -8)(-6.00000,9.00000){3}{\makebox(1.6667,11.6667){\tiny.}}
\multiput(1366, 10)(-6.00000,8.00000){3}{\makebox(1.6667,11.6667){\tiny.}}
\multiput(1354, 26)(-7.04920,8.45904){3}{\makebox(1.6667,11.6667){\tiny.}}
\multiput(1340, 43)(-6.43900,8.04875){3}{\makebox(1.6667,11.6667){\tiny.}}
\multiput(1327, 59)(-7.50000,7.50000){3}{\makebox(1.6667,11.6667){\tiny.}}
\multiput(1312, 74)(-7.75000,7.75000){3}{\makebox(1.6667,11.6667){\tiny.}}
\multiput(1297, 90)(-8.70490,7.25408){3}{\makebox(1.6667,11.6667){\tiny.}}
\multiput(1280,105)(-9.54100,7.95083){3}{\makebox(1.6667,11.6667){\tiny.}}
\multiput(1261,121)(-6.72130,5.60108){4}{\makebox(1.6667,11.6667){\tiny.}}
\multiput(1241,138)(-7.41333,5.56000){4}{\makebox(1.6667,11.6667){\tiny.}}
\multiput(1219,155)(-7.62667,5.72000){4}{\makebox(1.6667,11.6667){\tiny.}}
\multiput(1196,172)(-8.58667,6.44000){4}{\makebox(1.6667,11.6667){\tiny.}}
\multiput(1170,191)(-9.15383,6.10256){4}{\makebox(1.6667,11.6667){\tiny.}}
\multiput(1143,210)(-7.03845,4.69230){5}{\makebox(1.6667,11.6667){\tiny.}}
\multiput(1115,229)(-7.03845,4.69230){5}{\makebox(1.6667,11.6667){\tiny.}}
\multiput(1087,248)(-9.00000,6.00000){4}{\makebox(1.6667,11.6667){\tiny.}}
\multiput(1060,266)(-8.61540,5.74360){4}{\makebox(1.6667,11.6667){\tiny.}}
\multiput(1034,283)(-7.69607,4.61764){4}{\makebox(1.6667,11.6667){\tiny.}}
\multiput(1011,297)(-9.63235,5.77941){3}{\makebox(1.6667,11.6667){\tiny.}}
\multiput(992,309)(-13.84620,9.23080){2}{\makebox(1.6667,11.6667){\tiny.}}
\multiput(978,318)(-9.00000,6.00000){2}{\makebox(1.6667,11.6667){\tiny.}}
}%
{\color[rgb]{0,0,0}\multiput(1570,-490)(-8.00000,-10.00000){2}{\makebox(1.6667,11.6667){\tiny.}}
\multiput(1562,-500)(-6.00000,-7.50000){3}{\makebox(1.6667,11.6667){\tiny.}}
\multiput(1550,-515)(-5.28000,-7.04000){4}{\makebox(1.6667,11.6667){\tiny.}}
\multiput(1534,-536)(-6.28000,-8.37333){4}{\makebox(1.6667,11.6667){\tiny.}}
\multiput(1515,-561)(-5.43903,-6.79878){5}{\makebox(1.6667,11.6667){\tiny.}}
\multiput(1493,-588)(-5.67000,-7.56000){5}{\makebox(1.6667,11.6667){\tiny.}}
\multiput(1470,-618)(-6.21950,-7.77438){5}{\makebox(1.6667,11.6667){\tiny.}}
\multiput(1445,-649)(-6.25000,-7.50000){5}{\makebox(1.6667,11.6667){\tiny.}}
\multiput(1420,-679)(-6.02460,-7.22952){5}{\makebox(1.6667,11.6667){\tiny.}}
\multiput(1396,-708)(-5.67622,-6.81147){5}{\makebox(1.6667,11.6667){\tiny.}}
\multiput(1373,-735)(-7.10383,-8.52460){4}{\makebox(1.6667,11.6667){\tiny.}}
\multiput(1351,-760)(-6.80327,-8.16392){4}{\makebox(1.6667,11.6667){\tiny.}}
\multiput(1330,-784)(-6.66667,-6.66667){4}{\makebox(1.6667,11.6667){\tiny.}}
\multiput(1310,-804)(-6.33333,-6.33333){4}{\makebox(1.6667,11.6667){\tiny.}}
\multiput(1291,-823)(-6.16667,-6.16667){4}{\makebox(1.6667,11.6667){\tiny.}}
\multiput(1272,-841)(-9.00000,-7.50000){3}{\makebox(1.6667,11.6667){\tiny.}}
\multiput(1254,-856)(-8.70490,-7.25408){3}{\makebox(1.6667,11.6667){\tiny.}}
\multiput(1237,-871)(-8.88000,-6.66000){3}{\makebox(1.6667,11.6667){\tiny.}}
\multiput(1219,-884)(-9.00000,-6.00000){3}{\makebox(1.6667,11.6667){\tiny.}}
\multiput(1201,-896)(-9.00000,-6.00000){3}{\makebox(1.6667,11.6667){\tiny.}}
\multiput(1183,-908)(-9.60000,-4.80000){3}{\makebox(1.6667,11.6667){\tiny.}}
\multiput(1164,-918)(-9.77940,-5.86764){3}{\makebox(1.6667,11.6667){\tiny.}}
\multiput(1144,-929)(-9.80000,-4.90000){3}{\makebox(1.6667,11.6667){\tiny.}}
\multiput(1124,-938)(-7.75863,-3.10345){4}{\makebox(1.6667,11.6667){\tiny.}}
\multiput(1101,-948)(-11.46550,-4.58620){3}{\makebox(1.6667,11.6667){\tiny.}}
\multiput(1078,-957)(-8.60000,-2.86667){4}{\makebox(1.6667,11.6667){\tiny.}}
\multiput(1052,-965)(-9.00000,-3.00000){4}{\makebox(1.6667,11.6667){\tiny.}}
\multiput(1025,-974)(-9.90000,-3.30000){4}{\makebox(1.6667,11.6667){\tiny.}}
\multiput(995,-983)(-10.74510,-2.68628){4}{\makebox(1.6667,11.6667){\tiny.}}
\multiput(963,-992)(-8.29412,-2.07353){5}{\makebox(1.6667,11.6667){\tiny.}}
\multiput(930,-1001)(-8.70588,-2.17647){5}{\makebox(1.6667,11.6667){\tiny.}}
\multiput(895,-1009)(-8.70588,-2.17647){5}{\makebox(1.6667,11.6667){\tiny.}}
\multiput(860,-1017)(-8.47060,-2.11765){5}{\makebox(1.6667,11.6667){\tiny.}}
\multiput(826,-1025)(-10.27450,-2.56862){4}{\makebox(1.6667,11.6667){\tiny.}}
\multiput(795,-1032)(-9.35897,-1.87179){4}{\makebox(1.6667,11.6667){\tiny.}}
\multiput(767,-1038)(-11.53845,-2.30769){3}{\makebox(1.6667,11.6667){\tiny.}}
\multiput(744,-1043)(-8.47060,-2.11765){3}{\makebox(1.6667,11.6667){\tiny.}}
\multiput(727,-1047)(-12.00000,-2.00000){2}{\makebox(1.6667,11.6667){\tiny.}}
}%
\put(1426,119){\makebox(0,0)[lb]{\smash{{\SetFigFont{12}{14.4}{\rmdefault}{\mddefault}{\updefault}{\color[rgb]{0,0,0}a}%
}}}}
\put(3316,134){\makebox(0,0)[lb]{\smash{{\SetFigFont{12}{14.4}{\rmdefault}{\mddefault}{\updefault}{\color[rgb]{0,0,0}a}%
}}}}
\put(1021,-856){\makebox(0,0)[lb]{\smash{{\SetFigFont{12}{14.4}{\rmdefault}{\mddefault}{\updefault}{\color[rgb]{0,0,0}a}%
}}}}
\put(3511,-826){\makebox(0,0)[lb]{\smash{{\SetFigFont{12}{14.4}{\rmdefault}{\mddefault}{\updefault}{\color[rgb]{0,0,0}a}%
}}}}
\put(2176,-376){\makebox(0,0)[lb]{\smash{{\SetFigFont{12}{14.4}{\rmdefault}{\mddefault}{\updefault}{\color[rgb]{0,0,0}b}%
}}}}
\put(766,409){\makebox(0,0)[lb]{\smash{{\SetFigFont{12}{14.4}{\rmdefault}{\mddefault}{\updefault}{\color[rgb]{0,0,0}$\mu_1$}%
}}}}
\put(511,-1121){\makebox(0,0)[lb]{\smash{{\SetFigFont{12}{14.4}{\rmdefault}{\mddefault}{\updefault}{\color[rgb]{0,0,0}$\mu_2$}%
}}}}
\put(3896,409){\makebox(0,0)[lb]{\smash{{\SetFigFont{12}{14.4}{\rmdefault}{\mddefault}{\updefault}{\color[rgb]{0,0,0}$\mu_3$}%
}}}}
\put(4021,-1046){\makebox(0,0)[lb]{\smash{{\SetFigFont{12}{14.4}{\rmdefault}{\mddefault}{\updefault}{\color[rgb]{0,0,0}$\mu_4$}%
}}}}
\put(1531,-646){\makebox(0,0)[lb]{\smash{{\SetFigFont{12}{14.4}{\rmdefault}{\mddefault}{\updefault}{\color[rgb]{0,0,0}$\nu_1$}%
}}}}
\put(2831,-646){\makebox(0,0)[lb]{\smash{{\SetFigFont{12}{14.4}{\rmdefault}{\mddefault}{\updefault}{\color[rgb]{0,0,0}$\nu_2$}%
}}}}
\end{picture}%
\end{figure}

\appendix
\section{$L^\infty$ regularity and a stronger maximum principle}

Shortly after the initial posting of this paper, the author was made aware of a recent preprint of Agueh and Carlier \cite{AC10} in which the problem of minimizing 
\[
\Psi(\mu)=\sum_{i=1}^l \sigma_i W_2^2(\mu_i,\mu)
\]
over $\mu\in P_2(\R^n)$ is considered. In this appendix, we discuss briefly the applications of their work to minimal network problems. Specifically, we obtain $L^\infty$ estimates and derive a stronger maximum principle for minimal networks in $P_2(\R^n)$.

Building on the work of Gangbo and 
\`{S}wi\c{e}ch 
\cite{GS98}, Agueh and Carlier obtain the following $L^\infty$ regularity of solutions.

\begin{theorem}[Theorem 5.1 and Remark 5.2 of \cite{AC10}]\label{AC}
Let $\mu_1,\dots,\mu_l\in P_2(\R^n)$ and let $\sigma_1,\dots,\sigma_l$ be positive reals summing to $1/2$. Assume $\mu_1\in L^\infty$, i.e. $\mu_1=\rho_1 \mathcal{L}^n$ with $\rho_1\in L^\infty$.
If $\Psi(\nu)=\min_{\mu} \Psi(\mu)$, then 
\[
\|\nu\|_{L^\infty}
\le \frac1{\sigma_1^n} \|\mu_1\|_{L^\infty}.
\]
\end{theorem}

We may always rescale our coefficients so $\sum_i \sigma_i=1/2$, thereby obtaining such an estimate locally (on a star) for a minimal network. Moving inward from $\mu_i$ along stars, we obtain an interior $L^\infty$ estimate.

\begin{corollary}
Let $\Gamma:G\to P_2(\R^n)$ be a $G$-parameterized minimal network spanning $\mu_1,\dots,\mu_l\in P_2(\R^n)$. If $G$ has $k$ edges and the length of the image of each edge is in $[1/M,1/m]$, then for any $\nu\in \Gamma(G)$ with distance at least $1/M$ from a boundary vertex,
\[
\|\nu\|_{L^\infty}
\le \left(\frac{2kM}{m}\right)^{kn} \|\mu_1\|_{L^\infty}.
\]
\end{corollary}

For a global $L^\infty$ estimate, we may consider the above estimate away from the boundary combined with further applications of Theorem \ref{AC} on each leaf.
\begin{corollary}
Let $\Gamma:G\to P_2(\R^n)$ be a $G$-parameterized minimal network spanning $\mu_1,\dots,\mu_l\in P_2(\R^n)$. If $G$ has $k$ edges and the length of the image of each edge is in $[1/M,1/m]$, then for any $\nu\in \Gamma(G)$ and any $\lambda>1$,
\[
\|\nu\|_{L^\infty}
\le 
\left[\max\left\{\left(\frac{\lambda}{\lambda-1}\right)^n,\left(\frac{2kM}{\lambda m}\right)^{kn}\right\}\right] 
\left(\max_{i=1,\dots,l}\|\mu_i\|_{L^\infty}\right).
\]
\end{corollary}

In these estimates, we get control of $1/m$ and $k$ essentially for free, but we need some prior knowledge of a lower bound for $1/M$. Therefore, for a given set of boundary data, global $L^\infty$ estimates are reduced to a qualitative estimate on the non-degeneracy of edges.
Note also that the interior estimate may give information in cases where the global is trivial due to $\|\mu_i\|_{L^\infty}=\infty$ for some $i$.

Agueh and Carlier also define a notion of convexity along barycenters of a functional $\mathcal{F}:P_2(\R^n)\to\R$ by taking a weighted average of the values of $\mathcal{F}$ for comparison with $\mathcal{F}(\nu)$ for the barycenter, or minimizer of $\Psi$, $\nu$. They show that convexity along barycenters follows from convexity along generalized geodesics, and hence holds for the standard admissible functionals discussed above. Again, a simple induction along stars allows us to obtain a bound for $\mathcal{F}$ on a minimal network by a weighted average of the values of $\mathcal{F}$ on the boundary data, yielding a maximum principle as above. The weighted average will be less than the maximum of $\mathcal{F}$ on the boundary in general, so we obtain a stronger estimate in this way when the interior of $G$ avoids the boundary data.

\begin{theorem}
If $\mathcal{F}$ is convex along generalized geodesics in $P_2(\R^n)$, or even just convex along barycenters, and $\Gamma:G\to P_2(\R^n)$ is a $G$-parameterized minimal network spanning $\mu_1,\dots,\mu_l\in \mathcal{F}^{-1}(-\infty,m]$ for some $m<\infty$, then
\[
\Gamma(G)\subset\mathcal{F}^{-1}(-\infty,m].
\]
 Furthermore, if $\mathcal{F}(\nu)=m$ for some $\nu\in\Gamma(G)\setminus\{\mu_1,\dots,\mu_l\}$ and no interior vertices of $G$ are mapped to $\{\mu_1,\dots,\mu_l\}$, then $\Gamma(G)\subset\mathcal{F}^{-1}(m)$.
\end{theorem}

\bibliographystyle{plain}
\bibliography{max_princ_euc_was_steiner}

\end{document}